\documentclass  [12pt]{article} 
\usepackage{amsmath, amsfonts, amssymb, amsthm,natbib}

\theoremstyle{definition}
\newtheorem{theorem}{Theorem}[section]
\newtheorem{corollary}[theorem]{Corollary}
\newtheorem{lemma}[theorem]{Lemma}
\newtheorem{proposition}[theorem]{Proposition}

\newcommand{\be}{\begin{equation}}
\newcommand{\ee}{\end{equation}}
\newcommand{\beq}{\begin{equation*}}
\newcommand{\eeq}{\end{equation*}}
\newcommand{\bq}{\begin{eqnarray}}
\newcommand{\eq}{\end{eqnarray}}
\newcommand{\bqn}{\begin{eqnarray*}}
\newcommand{\eqn}{\end{eqnarray*}}

\newcommand{\inner}[2]{\left\langle{#1},{#2}\right\rangle}

\newcommand\ceil[1]{\lceil#1\rceil}
\newcommand\floor[1]{\lfloor#1\rfloor}

\def\E{{\mathbb{E}}\,}
\def\R{{\mathbb{R}}}

\def\P{{\mathbb{P}}}

\newcommand{\eps}{\varepsilon}

\newcommand{\EE}{\mathbb E}

\newcommand{\II}{\mathbb I}
\newcommand{\RR}{\mathbb R}
\newcommand{\NN}{\mathbb N}

\newcommand{\PP}{\mathbb P}

\newcommand{\cF}{{\mathcal F}}

\newcommand{\cM}{{\mathcal M}}

\newcommand{\cP}{{\mathcal P}}

\newcommand{\cR}{{\mathcal R}}
\newcommand{\cS}{{\mathcal S}}

\begin{document}

\title{How many Laplace transforms of probability measures are there?}
\author{{\sc Fuchang Gao}\thanks{Department of mathematics, University of Idaho, \texttt{fuchang@uidaho.edu}.} \and {\sc Wenbo V. Li\thanks{Department of Mathematical Sciences, University of Delaware,
\texttt{wli@math.udel.edu}. Supported in part by NSF grant
DMS-0805929. }} \and
 {\sc Jon A. Wellner}\thanks{Department of Statistics, University of Washington,
 \texttt{jaw@stat.washington.edu}. Supported in part by NSF Grant
 DMS-0804587}}
\date{\today}

\maketitle
\begin{abstract}
A bracketing metric entropy bound for the class of Laplace
transforms of probability measures on $[0,\infty)$ is obtained
through its connection with the small deviation probability of a
smooth Gaussian process.  
Our results for the particular smooth Gaussian process seem to be 
of  independent interest.
\end{abstract}

\bigskip



\bigskip
\bigskip

\newpage

\section{Introduction}
Let $\mu$ be a finite measure on $[0,\infty)$. The Laplace
transform of $\nu$ is a function on $(0,\infty)$ defined by \bq
f(t)=\int_0^\infty e^{-ty}\mu(dy).\label{Laplace}\eq It is easy to
check that such a function has the property that
$(-1)^nf^{(n)}(t)\ge 0$ for all non-negative integer $n$ and all
$t>0$. A function on $(0,\infty)$ with this property is called a
completely monotone function on $(0,\infty)$. A characterization
due to Bernstein (c.f. \cite{MR0077581}) says that $f$ is
completely monotone on $(0,\infty)$ if and only if there is a
non-negative measure $\mu$ (not necessary finite) on $[0,\infty)$ such that
(\ref{Laplace}) holds.
Therefore, due to monotonicity, the class of Laplace transforms of finite measures on
$[0,\infty)$ is the same as the class of bounded completely
monotone functions on $(0,\infty)$. These functions can
be extended to continuous functions on $[0,\infty)$, and we will
call them completely monotone on $[0,\infty)$.

Completely monotonic functions have remarkable applications in
various fields, such as probability and statistics, physics and
potential theory. The main properties of these functions are given
in \cite{MR0005923}, Chapter IV. For example, the class of
completely monotonic functions is closed under sums, products and
pointwise convergence. We refer to \cite{MR1922200} for a detailed
list of references on completely monotonic functions. Closely
related to the class of completely monotonic functions are the
so-called $k$-monotone functions, where the non-negativity of
$(-1)^nf^{(n)}$ is required for all integers $n\le k$. In fact,
completely monotonic functions can be viewed as the limiting case
of the $k$-monotone functions as $k\to\infty$. In this sense, the
present work is a partial extension of \cite{MR2386068} and
\cite{GW_Science}. 

Let $\cM_\infty$ be the class of completely monotone functions on
$[0,\infty)$ that are bounded by $1$. Then 
\bqn 
\cM_\infty
&=&\left\{f: [0,\infty)\to [0,\infty) \left|f(t)=\int_0^\infty
e^{-tx}\mu(dx), \|\mu\|\le 1\right.\right\}.
\eqn
It is well known (see e.g. \cite{MR0270403}, Theorem 1, page 439) that 
the sub-class of ${\cM}_{\infty}$ with $f(0)=1$ corresponds exactly to the 
Laplace transforms of the
class of probability measures $\mu$ on $[0,\infty)$.  
For a random variable with distribution function $F(t) = P(X \le t)$, 
the {\sl survival function} $S(t) = 1-F(t) = P(X> t)$.
Thus the class
\bqn
\cS_{\infty} = \left\{S: [0,\infty)\to [0,\infty) \left| S(t)=\int_0^\infty
e^{-tx}\mu(dx), \|\mu\| = 1\right.\right\}
\eqn
is exactly the class of survival functions of all scale mixtures of the standard 
exponential distribution (with survival function $e^{-t}$), 
with corresponding densities 
\bqn
p(t) = -S'(t) = \int_0^\infty x e^{-x t} \mu (dx) , \qquad t \ge 0 .
\eqn
It is easily seen that the class $\cP_{\infty}$ of such densities with $p(0) < \infty$ 
is also a class of completely monotone functions corresponding to probability 
measures $\mu $ on $[0,\infty)$ with finite first moment.   
These classes have many applications in statistics; see e.g. 
\cite{MR653523} for a brief survey.   
\cite{MR653523} considered nonparametric estimation of a completely monotone density
and showed that the nonparametric maximum likelihood estimator (or MLE) for this class is 
almost surely consistent.  
The bracketing entropy bounds derived below can be considered as a first step 
toward global rates of convergence of the MLE.

In probability and statistical applications, one way to understand the 
complexity of a function class  is by way of 
the metric entropy for the class under certain common distances.
Recall that the metric entropy of a function class $\cF$ under
distance $\rho$ is defined to be $\log N(\eps,\cF,\rho)$ where
$N(\eps,\cF,\rho)$ is the minimum number of open balls of radius
$\eps$ needed to cover $\cF$. In statistical applications,
sometimes we also need bracketing metric entropy which is defined
as $\log N_{[\,]}(\eps,\cF,\rho)$ where
$$
N_{[\,]}(\eps, \cF,\rho):=\min \left\{n: \exists
\underline{f}_1,\overline{f}_1, \dots,
\underline{f}_n,\overline{f}_n  \mbox { s.t. }
\rho(\overline{f}_k,\underline{f}_k)\le \eps, \cF\subset
\bigcup_{k=1}^n[\underline{f}_k,\overline{f}_k]\right\},
$$
and
$$
[\underline{f}_k,\overline{f}_k]=\left\{g\in \cF:
\underline{f}_k\le g\le \overline{f}_k\right\}.
$$
Clearly $N(\eps,\cF,\rho)\le N_{[\,]}(\eps,\cF,\rho)$ and they are close related in our setting below.

In this paper, we study the metric entropy of $\cM_\infty$ under the 
$L^p(\nu)$ norm given by
$$\|f\|_{L^p(\nu)}^p=\int_0^\infty |f(x)|^p \nu(dx), \quad 1 \le p \le \infty,
$$
where $\nu$ is a probability measure on $[0,\infty)$.
Our main result is the following
\begin{theorem}\label{theorem}
(i) Let $\nu$ be a probability measure on $[0,\infty)$. There exists
a constant $C$ depending only on $p\ge 1$ such that for any
$0<\eps<1/4$,
$$\log N_{[\,]}(\eps,\cM_\infty,\|\cdot\|_{L^p(\nu)})\le C\log (\Gamma/\gamma)\cdot|\log \eps|^2,$$
for any $0<\gamma<\Gamma<\infty$ such that $\nu([\gamma, \Gamma])\ge
1-4^{-p}\eps^p$. In particular, if there exists a constant $K>1$,
such that $\nu([\eps^K, \eps^{-K}])\ge 1-4^{-p}\eps^p$, then
$$\log N_{[\,]}(\eps,\cM_\infty,\|\cdot\|_{L^p(\nu)})\le CK|\log\eps|^3 .$$

(ii) If $\nu$ is  Lebesgue measure on $[0,1]$, then
$$\log N_{[\,]}(\eps,\cM_\infty,\|\cdot\|_{L^2(\nu)})\asymp \log N(\eps,\cM_\infty,\|\cdot\|_{L^2(\nu)})\asymp |\log \eps|^3 ,$$
where $A\asymp B$ means there exist universal constants $C_1,
C_2>0$ such that $C_1A\le B \le C_2B$.
\end{theorem}

As an equivalent result for part (ii) of the above theorem,  we have the following important small
deviation probability estimates for an associated smooth Gaussian process. In particular, it may be
of interest to find a probabilistic proof for the lower bound directly.

\begin{theorem}\label{cor2}
Let $Y(t)$, $t>0$, be a Gaussian process with covariance 
$\E Y(t)Y(s)=(1-e^{-t-s})/(t+s)$, then for $0<\eps<1$
$$
\log \P\left(\sup_{t>0}|Y(t)|<\eps\right)\asymp -|\log\eps|^3.
$$
\end{theorem}

The rest of the paper is organized as follows. In Section 2, we provide 
the upper bound estimate in the main result by explicit construction.
In Section 3, we summarize various connections between entropy numbers of a set (and its convex hull)
and small ball probability for the associated Gaussian process.
Some of our observations in a general setting are stated explicitly for the first time.
Finally we identify the particular Gaussian process suitable for our entropy estimates.
Then in Section 4, we obtain the required upper bound small ball probability estimate
(which implies the lower bound entropy estimates as discussed in section 3) by a simple determinant estimates.
This method of small ball estimates is made explicit here for the first time and can be used in many more problems.
The technical determinant estimates are also of independent interests.

\section{Upper Bound Estimate}
In this section, we provide an upper bound for
$N_{[\,]}(\eps,\cM_\infty,\|\cdot\|_{L^p(\nu)})$, where $\nu$ is a probability measure on $[0,\infty)$ and $1 \le p \le \infty$.
This is accomplished by an explicit construction of $\eps$-brackets under $L^p(\nu)$ distance.

For each $0<\eps<1/4$, we choose $\gamma>0$ and $\Gamma=2^m\gamma$
where $m$ is a positive integer such that 
$\nu([\gamma,\Gamma]) \ge 1-4^{-p}\eps^p$. 
We use the notion $\II(a \le t < b)$
to denote the indicator function of the interval $[a,b)$. Now for
each $f\in \cM_\infty$, we first write in block form 
\bqn 
f(t) = \II(0 \le t < \gamma) f(t)+\II(t\ge \Gamma)f(t) +\sum_{i=1}^m
\II(2^{i-1}\gamma\le t< 2^i \gamma)f(t). 
\eqn 
Then for each block
$2^{i-1}\gamma\le t< 2^i \gamma$, we separate the integration
limits at the level $2^{2-i}|\log \eps|/\gamma$ and use the first
$N$ terms of Taylor's series expansion of $e^{-u}$ with error
terms associated with $\xi=\xi_{u, N}$, $0 \le \xi \le 1$, to
rewrite
$$
f(t) = \II(0 \le t < \gamma) f(t)+\II(t\ge \Gamma)f(t) +\sum_{i=1}^m (p_i(t)+q_i(t)+r_i(t))
$$
where
\bqn
p_i(t)&=:& \II(2^{i-1}\gamma\le t< 2^i \gamma)
\sum_{n=0}^{N}\frac{(-1)^nt^n }{n!}\int_0^{2^{2-i}|\log \eps|/\gamma}x^n\mu(dx)\\
q_i(t)&=:& \II(2^{i-1}\gamma\le t< 2^i \gamma)
\int_0^{2^{2-i}|\log \eps|/\gamma} \frac{(-\xi tx)^{N+1}}{(N+1)!}\mu(dx)\\
r_i(t)&=:& \II(2^{i-1}\gamma\le t< 2^i \gamma)\int_{2^{2-i}|\log \eps|/\gamma}^\infty e^{-tx}\mu(dx).
\eqn
We choose the integer $N$ so that
\bq
4e^2|\log \eps|-1\le N<4e^2|\log \eps|. \label{2.1-1}
\eq
Then, by using the inequality $k!\ge (k/e)^k$ and the fact that $0<\xi<1$, we have
within the block $2^{i-1}\gamma\le t< 2^i\gamma$,
\bqn
|q_i(t)|
&\le& \int_0^{2^{2-i}|\log \eps|/\gamma} \frac{(tx)^{N+1}}{(N+1)!}\mu(dx)\\
&\le& \frac{|4\log \eps|^{N+1}}{(N+1)!} \le \left(\frac{4e|\log \eps|}{N+1}\right)^{N+1}\le e^{-(N+1)} \le \eps^{4e^2},
\eqn
where we used $t x \le 2^i\gamma \cdot 2^{2-i}|\log \eps|/\gamma =4|\log \eps|$ in the second inequality above.
This implies, due to disjoint supports of $q_i(t)$,
\bq\left|\sum_{i=1}^mq_i(t)\right|\le\eps^{4e^2}.\label{q}\eq
Next, we notice that for $t\ge 2^{i-1}\gamma$ and $x\ge
2^{2-i}\gamma^{-1}|\log \eps|$, $e^{-tx}\le \eps^2$.
Thus \bq
\left|\sum_{i=1}^mr_i(t)\right|\le\sum_{i=1}^m \II(2^{i-1}\gamma\le
t< 2^i \gamma)\int_{2^{2-i}\gamma^{-1}|\log \eps|}^\infty \eps^2\mu(dx)
\le \eps^2.\label{r} \eq
Finally, because $|f|\le 1$ and
$\nu([0,\gamma))+\nu([\Gamma,\infty))\le 4^{-p}\eps^p$, we have
$$\left\|1_{0\le t<\gamma}f(t)+1_{t\ge \Gamma}f(t)\right\|_{L^p(\nu)}\le
\eps/4.$$ Together with (\ref{q}) and (\ref{r}), we see that the
set
$$\cR=:\left\{\sum_{i=1}^mq_i(t)+\sum_{i=1}^mr_i(t)+\II(t< \gamma)f(t)+\II(t\ge \Gamma)f(t): f\in \cM_\infty\right\}$$
has diameter in $L^p(\nu)$-distance at most
$\eps^2+\eps^{4e^2}+\eps/4<\eps/2$.

Therefore, if we denote $\cP_i=\{p_i(t): f\in \cM_\infty\}$, then
the expansion of $f$ above implies that $\cM_\infty\subset
\sum_{i=1}^m\cP_i+\cR$, and consequently, we have \bqn
N_{[\,]}(\eps, \cM_\infty,\|\cdot\|_{L^p(\nu)})\le
N_{[\,]}\left(\eps/2,
\sum_{i=1}^m\cP_i,\|\cdot\|_{L^p(\nu)}\right). \eqn For any $1\le
i\le m$ and any $p_i\in \cP_i$, we can write \bq
p_i(t)=\II(2^{i-1}\gamma\le t< 2^i
\gamma)\sum_{n=0}^{N_i}(-1)^na_{ni} (2^{-i} \gamma^{-1}t)^n,
\label{2.4} \eq where $0\le a_{ni}\le |4\log \eps|^n/n!$. Now we
can construct \bqn \overline{p}_i=\II(2^{i-1}\gamma\le t< 2^i
\gamma)\sum_{n=0}^{N}(-1)^nb_{ni} (2^{-i} \gamma^{-1}t)^n, \\
\underline{p}_i=\II(2^{i-1}\gamma\le t< 2^i
\gamma)\sum_{n=0}^{N}(-1)^nc_{ni} (2^{-i} \gamma^{-1}t)^n,\eqn
where
$$b_{ni}=\left\{\begin{array}{ll}\frac{\eps}{2^{n+2}}\ceil{\frac{2^{n+2}a_{ni}}{\eps}}& n \mbox{ is
even}\\\\
\frac{\eps}{2^{n+2}}\floor{\frac{2^{n+2}a_{ni}}{\eps}}& n \mbox{
is odd}\end{array}
\right.,\,\,\,\,\,\,c_{ni}=\left\{\begin{array}{ll}\frac{\eps}{2^{n+2}}\floor{\frac{2^{n+2}a_{ni}}{\eps}}&
n \mbox{ is
even}\\\\
\frac{\eps}{2^{n+2}}\ceil{\frac{2^{n+2}a_{ni}}{\eps}}& n \mbox{ is
odd}\end{array} \right..$$ Clearly, $\underline{p}_i(t)\le p_i(t)\le
\overline{p}_i(t)$, and
\bqn |\overline{p}_i-\underline{p}_i|&\le&
\II(2^{i-1}\gamma\le t< 2^i \gamma)\sum_{n=0}^{N}|c_{ni}-b_{ni}|(2^{-i}
\gamma^{-1}t)^n\\
&\le& \II(2^{i-1}\gamma\le t< 2^i \gamma)
\sum_{n=0}^{N}\frac{\eps}{2^{n+2}}(2^{-i} \gamma^{-1}t)^n\\
&\le& \frac{\eps}{2}\II(2^{i-1}\gamma\le t< 2^i \gamma).\eqn
Hence
$$
\sum_{i=1}^m \underline{p}_i\le \sum_{i=1}^mp_i\le \sum_{i=1}^m \overline{p}_i\le \sum_{i=1}^m \underline{p}_i+\eps/2.
$$
That is, the sets
$$
\underline{\cP}=:\left\{\sum_{i=1}^m \underline{p}_i: p_i\in \cP_i, 1\le i\le m\right\}\ \mbox{ and } \ 
\overline{\cP}=:\left\{\sum_{i=1}^m \overline{p}_i: p_i\in \cP_i, 1\le i\le m\right\}
$$
form $\eps/2$ brackets of $\sum_{i=1}^m\cP_i$ in $L^\infty$-norm,
and thus in $L^p(\nu)$-norm for all $1\le p<\infty$.

 Now we count the number of different realizations of
$\overline{\cP}$ and $\underline{\cP}$. 
Note that, due to the uniform bound on $a_{ni}$ in (\ref{2.4}) there are no
more than $$\frac{2^{n+1}}{\eps}\cdot
\frac{|4\log \eps|^n}{n!}+1$$ realizations for $b_{ni}$. 
So,
the number of realizations of $\overline{p}_i$ is bounded by
\bqn
\prod_{n=0}^{N}\left(\frac{2^{n+1}}{\eps}\cdot
\frac{|4\log \eps|^n}{n!}+1\right). \eqn
Because $n! >(n/e)^n$, for all $1\le n\le N$,  we have \bqn \frac{2^{n+1}}{\eps}\cdot
\frac{|4\log \eps|^n}{n!}+1\le
\frac3\eps\left(\frac{8e|\log \eps|}{n}\right)^n .\eqn
Thus, the
number of realizations of $\overline{p}_i$ is bounded by \bqn
&&\left(\frac{3}{\eps}\right)^{N+1}\cdot
\exp\left(\sum_{n=1}^{N} (n \log |8e\log \eps|-n \log n)\right)\\
&\le& \left(\frac{3}{\eps}\right)^{N+1}\cdot
\exp\left(\frac{N(N+1)}{2}\log |8e\log \eps|-\int_1^{N}x\log xdx\right)\\
&\le& \left(\frac{3}{\eps}\right)^{N+1}\cdot
\exp\left(\frac{N(N+1)}{2}\log |8e\log \eps|-\frac{N^2}{2}\log N+\frac{N^2}{4}\right)\\
&\le& \exp\left(C|\log \eps|^2\right)\eqn for some absolute
constant $C$, where in the last inequality we used the bounds on $N$ given in (\ref{2.1-1}).

Hence the total number of realizations of $\overline{\cP}$ is
bounded by $\exp\left(Cm|\log \eps|^2\right)$. Similar
estimate holds for the total number of realizations of
$\underline{\cP}$, and we finally obtain
$$\log N_{[\,]}(\eps,
\cM_\infty,\|\cdot\|_{L^p(\nu)})\le C'm |\log\eps|^2
$$for some different constant $C'$.
This finishes the proof since $m=\log_2(\Gamma/\gamma)$.

\section{Entropy of Convex Hulls}
A lower bound estimate of metric entropy is typically difficult,
because it often involves a construction of a well-separated set
of maximal cardinality.  Thus we introduce some soft
analytic arguments to avoid this difficulty and change the problem into
a familiar one in this section. The hard estimates are given in the next section.

First note that $\cM_\infty$ is just the convex hull of the functions
$k_s(\cdot)$, $0<s<\infty$, where $k_s(t)=e^{-ts}$. We recall a
general method about the entropy of convex hulls that was
introduced in \cite{MR2069008}. Let $T$ be a set in $\R^n$ or in a
Hilbert space. The convex hull of $T$ can be expressed as
$${\rm conv}(T)=\left\{\sum_{n=1}^{\infty}a_nt_n: t_n\in T, a_n\ge 0,  n\in \NN,
\sum_{n=1}^{\infty}a_n=1\right\};$$ while the absolute convex hull
of $T$ is defined by
$${\rm abconv}(T)=\left\{\sum_{n=1}^{\infty}a_nt_n: t_n\in T, n\in \NN,
\sum_{n=1}^{\infty}|a_n|\leq 1\right\}.$$ Clearly, by using
probability measures and signed measures, we can express \bqn
&&{\rm conv}(T)=\left\{\int_T t\mu(dt): \mu \mbox{ is a
probability measure on } T \right\};\\
&&{\rm abconv}(T)=\left\{\int_T t\mu(dt): \mu \mbox{ is a signed
measure on } T, \|\mu\|_{TV}\leq 1 \right\}.\eqn For any norm
$\|\cdot\|$ the following is clear:
$${\rm conv}(T)\subset {\rm abconv}(T)\subset {\rm conv}(T)-{\rm conv}(T).$$
Therefore, $$N(\eps,{\rm conv}(T),\|\cdot\|)\le N(\eps,{\rm
abconv}(T),\|\cdot\|)\le [N(\eps/2,{\rm conv}(T),\|\cdot\|)]^2.$$
In particular, at the logarithmic level, the two entropy numbers are comparable,
modulo constant factors on $\eps$.
The benefit of using absolute convex hull is that it is symmetric
and can be viewed as the unit ball of a Banach space, which
allows us to use the following duality lemma of metric entropy:
$$\log N(\eps,{\rm abconv}(T),\|\cdot\|)\asymp \log N(c_2\eps,B,\|\cdot\|_T)$$
where $B$ is dual ball of the norm $\|\cdot\|$, and $\|\cdot\|_T$
is the norm introduced by $T$, that is,
$$\|x\|_T:=\sup_{t\in T}|\inner{t}{x}|=\sup_{t\in {\rm abconv}(T)}|\inner{t}{x}|.$$
Strictly speaking, the duality lemma remains as a conjecture in
the general case. However, when the norm $\|\cdot\|$ is the
Hilbert space norm, this has been proved. See \cite{MR910364},
\cite{MR1008716}, and \cite{MR2105957}.

A striking relation discovered by \cite{MR1237989} says that the
entropy number $\log N(\eps,B,\|\cdot\|_T)$ is determined by the
Gaussian measure of the set 
$$
D_\eps=:\{x\in H: \|x\|_T\leq \varepsilon \}
$$ 
under some very weak regularity assumptions. For
details, see \cite{MR1237989},  \cite{MR1733160}, and also
Corollary 2.2 of \cite{Aurzada}. Using this relation, we can now
summarize the connection between metric entropy of convex hulls
and Gaussian measure of $D_\eps$ into the following
\begin{proposition}\label{prop}
Let $T$ be a precompact set in a Hilbert space. For $\alpha>0$ and
$\beta\in \RR$,
$$
\log \PP\left(D_\eps\right)\le -C_1\eps^{-\alpha}|\log\eps|^\beta
$$
if and only if 
$$
\log N(\eps,{\rm conv}(T), \|\cdot\|)\ge
C_2\eps^{-\frac{2\alpha}{2+\alpha}}|\log \eps|^{\frac{2\beta}{2+\alpha}};
$$
and for $\beta>0$ and $\gamma\in \RR$,
$$
\log \PP\left(D_\eps\right)\le -C_1|\log \eps|^\beta (\log |\log \eps|)^\gamma
$$
if and only if 
$$ 
\log N(\eps,{\rm conv}(T), \|\cdot\|_2)\ge C_2
|\log\eps|^\beta (\log |\log \eps|)^\gamma .
$$ 
Furthermore, the
results also hold if the directions of the inequalities are switched.
\end{proposition}
The result of this proposition can be implicitly seen in
\cite{MR2069008}, where an explanation of the relation between
$N(\eps,B,\|\cdot\|_T)$ and the Gaussian measure of $D_\eps$ is
also given.

Perhaps, the most useful case of Proposition~\ref{prop} is when
$T$ is a set of functions: $K(t,\cdot)$, $t\in T$, where for each
fixed $t\in T$, $K(t,\cdot)$ is a function in $L^2(\Omega)$, and
where $\Omega$ is a bounded set in $\R^d$, $d\ge 1$. For this
special case, we have
\begin{corollary}\label{cor}
Let $X(t)=\int_{\Omega}K(t,x)dB(x)$, $t\in T$, where $K(t,\cdot)$
are square-integrable functions on a bounded set $\Omega$ in
$\R^d$, $d\ge 1$, and $B(x)$ is the $d$-dimensional Brownian sheet
on $\Omega$. If $\cF$ is the convex hull of the functions
$K(\cdot,\omega)$, $\omega\in \Omega$, then
$$
\log \PP\left(\sup_{t\in T}|X(t)|<\eps\right)\asymp \eps^{-\alpha}|\log\eps|^\beta
$$
for $\alpha>0$ and $\beta\in \RR$ if and only if 
$$\log
N(\eps,\cF, \|\cdot\|)\asymp \eps^{-\frac{2\alpha}{2+\alpha}}|\log \eps|^{\frac{2\beta}{2+\alpha}};
$$
and for $\beta>0$ and $\gamma\in
\RR$,
$$
\log \PP\left(\sup_{t\in T}|X(t)|<\eps\right)\asymp-|\log \eps|^\beta (\log |\log \eps|)^\gamma
$$
if and only if 
$$ \log N(\eps,\cF, \|\cdot\|_2)\asymp
|\log\eps|^\beta (\log |\log \eps|)^\gamma .
$$
\end{corollary}
The authors found this corollary especially useful. For example,
it was used in \cite{MR2341952} and \cite{MR2386068} to change a
problem of metric entropy to a problem of small deviation
probability of a problem about a Gaussian process which is
relatively easier. The proof is given in \cite{MR2386068} for the
case $\Omega=[0,1]$, and in \cite{MR2341952} for the case
$[0,1]^d$. For the general case, it can be proved as easily.
Indeed, the only thing we need to prove is that $\P(D_\eps)$ can
be expressed as the probability of the set $\sup_{t\in T}|X(t)|<
\eps$. We outline a proof below. Let $\phi_n$ be an orthonormal
basis of $L^2(\Omega)$, then
$$X(t)=\int_\Omega K(t,s)dB(s)=\sum_{n=1}^\infty \xi_n\int_\Omega K(t,s)\phi_n(s)ds$$
where $\xi_n$ are i.i.d standard normal random variables. Thus,
\bqn \P(D_\eps)&=&\P\left\{g\in L^2(\Omega): \left|\int_\Omega
f(s)g(s)ds\right|< \eps, f\in \cF\right\}\\
&=&\P\left\{g\in L^2(\Omega): \left|\int_T\int_{\Omega}
K(t,s)g(s)ds\mu(dt)\right|< \eps, \|\mu\|_{TV}\le 1\right\}\\
&=&\P\left\{\sum_{n=1}^\infty a_n\phi_n(s): \sum_{n=1}^\infty
a_n^2<\infty, \left|\sum_{n=1}^\infty a_n\int_T\int_{\Omega}
K(t,s)\phi_n(s)ds\mu(dt)\right|< \eps, \|\mu\|_{TV}\le 1\right\}\\
&=&\P\left\{\sum_{n=1}^\infty a_n\phi_n(s): \sum_{n=1}^\infty
a_n^2<\infty,  \sup_{t\in T}\left|\sum_{n=1}^\infty
a_n\int_{\Omega} K(t,s)\phi_n(s)ds\right|< \eps\right\}\\
&=&\P\left(\sup_{t\in T}|X(t)|<\eps\right).
 \eqn

Now back to our problem of estimate
 $\log
N(\eps,\cM_\infty,\|\cdot\|_2)$ in the statement of (ii) of the
theorem, where $\|\cdot\|_2$ is the $L^2$ norm under the Lebesgue
measure on $[0,1]$, we notice that $\cM_\infty$ is the convex hull
of the functions $C(\cdot,s)$, $s\in [0,\infty)$, on $[0,1]$ with
$C(t,s)=e^{-ts}$. However, $[0,\infty)$ is not bounded. In order
to use Corollary~\ref{cor}, we need to make a change of variables.
Notice that by letting $y=e^{-s}$, we can view $\cM_\infty$ as
convex hull of $K(\cdot, y)$, $y\in (0,1]$, where $K(t,y)=y^t$.
Clearly, $K(t,\cdot)$ are square-integrable functions on the
bounded set $(0,1]$. Now, for this $K$, the corresponding $X(t)$
is a Gaussian process on $[0,1]$ with covariance
$$
\E X(t)X(s)=\frac{ts-1}{\log(ts)}, \quad s,t\in (0,1], \quad (s,t)\ne (1,1),
$$
and $\E X(1)^2=1$. Thus, the problem becomes how to estimate
$$\P\left(\sup_{t\in (0,1]}|X(t)|<\eps\right),\ \mbox{or
equivalently} \ \P\left(\sup_{t\ge 0}|Y(t)|<\eps\right),$$ where
$Y(t)=X(e^{-t})$, which has covariance structure
\bq
\E Y(t)Y(s)=\frac{1-e^{-t-s}}{t+s}, \quad s, t \ge 0. \label{3.2}
\eq
We now turn to the lower estimates of this probability.

\section{Lower Bound Estimate} Let $Y(t)$, $t \ge 0$ be the centered Gaussian process defined in (\ref{3.2}).
Our goal in this section is to prove that
$$
\log \PP(\sup_{t\ge 0}|Y(t)|<\eps)\le -C |\log \eps|^3,
$$
for some constant $C>0$.

Note that for any sequence of positive numbers
$\{\delta_i\}_{i=1}^n$,
\begin{eqnarray}
\PP\left(\sup_{t\ge 0}|Y(t)|<\eps\right)
&\le & \PP(\max_{1 \le i \le n} |Y( \delta_i)|<\eps) \nonumber\\
&=& (2\pi)^{-n/2} (\det \Sigma)^{-1/2} \int_{ \max_{1 \le i \le n}
|y_i| \le \eps}
\exp \left( -\langle y, \Sigma^{-1} y \rangle \right) dy_1\cdots dy_n \nonumber\\
&\le& (2\pi)^{-n/2} (\det \Sigma)^{-1/2} (2\eps)^n \nonumber \\
&\le & \eps^n (\det \Sigma)^{-1/2}.\label{determinant}
\end{eqnarray}
where the covariance matrix
$$
\Sigma= \left( \EE Y(\delta_i)Y(\delta_j) \right)_{1 \le i,j \le
n} =\left( { 1-e^{-\delta_i-\delta_j} \over {\delta_i+\delta_j}
}\right)_{1 \le i,j \le n}.
$$

To find a lower bound for $\det(\Sigma)$, we need the following
lemma:
\begin{lemma}\label{lemma}
If $0<b_{ij}<a_{ij}$ for all $1\le i,j\le n$ then
$$\det(a_{ij}-b_{ij})\ge \det(a_{ij})-\sum_{k=1}^n\max_{1\le l\le n}\frac{b_{kl}}{a_{kl}}\cdot {\rm
per}(a_{ij}).$$
where ${\rm per}(a_{ij})$ is the permanent of the matrix $(a_{ij})$.
\end{lemma}
\begin{proof} For notational simplicity, we denote
$c_{ij}=a_{ij}-b_{ij}$, then \bqn
&&\det(a_{ij}-b_{ij})-\det(a_{ij})\\&=&\sum_{\sigma}(-1)^\sigma
c_{1,\sigma(1)}c_{2,\sigma(2)}c_{n,\sigma(n)}-\sum_{\sigma}(-1)^\sigma
a_{1,\sigma(1)}a_{2,\sigma(2)}\cdots a_{n,\sigma(n)}\\
&=&\sum_{\sigma}(-1)^\sigma \sum_{k=1}^n [c_{1,\sigma(1)}\cdots
c_{k-1,\sigma(k-1)}](c_{k,\sigma(k)}-a_{k,\sigma(k)})[
a_{k+1,\sigma(k+1)}\cdots a_{n,\sigma(n)}]\\
&\ge& -\sum_{\sigma} \sum_{k=1}^n [a_{1,\sigma(1)}\cdots
a_{k-1,\sigma(k-1)}](b_{k,\sigma(k)})[ a_{k+1,\sigma(k+1)}\cdots
a_{n,\sigma(n)}]\\
&\ge&  -\sum_{k=1}^n \max_{1\le l\le n}\frac{b_{kl}}{a_{kl}}
\sum_{\sigma}  [a_{1,\sigma(1)}\cdots
a_{k-1,\sigma(k-1)}](a_{k,\sigma(k)})[ a_{k+1,\sigma(k+1)}\cdots
a_{n,\sigma(n)}]\\
&=& - \sum_{k=1}^n\max_{1\le l\le n}\frac{b_{kl}}{a_{kl}}\cdot {\rm
per}(a_{ij}).\eqn
\end{proof}
In order to use Lemma~\ref{lemma} to estimate $\det(\Sigma)$, we set
$$
a_{ij}=\frac{1}{\delta_i+\delta_j}, \quad
{\rm and} \quad  b_{ij}=e^{-\delta_i-\delta_j}a_{ij}
$$
for a specific sequence
$\{\delta_i\}_{i=1}^n$ defined by
$$\delta_{mp+q}=4^{p+m}(m+q), \quad 0 \le p<m, 1 \le q \le m
$$
for $n=m^2$.

Clearly, we have
\bq
0<b_{kl}/a_{kl}\le e^{-2m4^m}, \quad 1\le k,l\le n=m^2. \label{4.2}
\eq
It remains to estimate $\det(a_{ij})$ and ${\rm
per}(a_{ij})$, which is given in the following lemma.

\begin{lemma}\label{lemma2}
For the matrix $(a_{ij})$ defined above, we have ${\rm
per}(a_{ij})\le 1$, and $\det(a_{ij})\ge (240e)^{-2m^3}$.
\end{lemma}
\begin{proof} It is easy to see
$${\rm per}(a_{ij})\le n!(\max_{i,j}a_{ij})^n\le \frac{(m^2)!}{(2m4^m)^{m^2}}\le 1$$
since $a_{ij}\le(2m4^m)^{-1}$ for $1\le i,j \le n=m^2$.

To estimate $\det(a_{ij})$, we use the Cauchy's determinant
identity, see \cite{MR1701596},
$$\det(a_{ij})=\det\left(\frac1{\delta_i+\delta_j}\right)=\frac{\prod_{1\le i<j\le n}(\delta_j-\delta_i)^2}{\prod_{1\le i,j\le
n}(\delta_j+\delta_i)}
=\frac{1}{2^n\prod_{i=1}^n\delta_i}\cdot\prod_{1\le i<j\le
n}\left(\frac{\delta_j-\delta_i}{\delta_j+\delta_i}\right)^2.$$

For $1\le i<j\le n=m^2$, write $i=mp+q$ and $j=mr+s$ with $1 \le q, s \le m$.
Denote
\bqn
&&A=\{(i,j): i=mp+q, j=mp+s, 0\le p\le m-1, 1\le q<s\le m\},\\
&&B=\{(i,j): i=mp+q, j=m(p+1)+s, 0\le p\le m-2, 1\le q, s\le
m\},\\
&&C=\{(i,j): i=mp+q, j=mr+s, 0\le p\le m-3, p+2\le r\le m-1, 1\le
q,s\le m\}.\eqn
Then $A$, $B$ and $C$ form a partition of $1\le i<j\le n=m^2$.

Next we estimate each part separately. First, for $p=r$,
$$\frac{\delta_j-\delta_i}{\delta_j+\delta_i}
=\frac{s-q}{2m+s+q}>\frac{s-q}{4m}.$$ Thus
\bqn
\prod_{(i,j)\in A}\left(\frac{\delta_j-\delta_i}{\delta_j+\delta_i}\right)^2
&\ge& \prod_{p=0}^{m-1}\prod_{1\le q<s\le m} \left(\frac{s-q}{4m}\right)^2
= \prod_{k=1}^{m-1} \prod_{q=1}^{m-k} \left(\frac{k}{4m}\right)^{2m} \\
&\ge & \prod_{k=1}^{m-1}  \left(\frac{k}{4m}\right)^{2m^2}
= \left(\frac{(m-1)!}{(4m)^{m-1}}\right)^{2m^2}\\
&\ge& (8e)^{-2m^3}.
\eqn
Second, for $r-p=1$,
$$\frac{\delta_j-\delta_i}{\delta_j+\delta_i}
=\frac{(4m+4s)-(m+q)}{(4m+4s)+(m+q)}\ge \frac{1}{5}.$$
Thus we have
$$\prod_{(i,j)\in B}\left(\frac{\delta_j-\delta_i}{\delta_j+\delta_i}\right)^2\ge \prod_{p=0}^{m-2}\prod_{1\le q, s\le m}
5^{-2}\ge 5^{-2m^3}.$$
Third,  for $r-p\ge 2$,
$$\left(\frac{\delta_j-\delta_i}{\delta_j+\delta_i}\right)
=\frac{4^r(m+s)-4^p(m+q)}{4^r(m+s)+4^p(m+q)}=1-\frac{2\cdot
4^{p}(m+q)}{4^r(m+s)+4^p(m+q)}>1-\frac{1}{4^{r-p-1}}.$$
Thus we have
\bqn
\prod_{(i,j)\in
C}\left(\frac{\delta_j-\delta_i}{\delta_j+\delta_i}\right)^2&\ge&
\prod_{p=0}^{m-3}\prod_{r=p+2}^{m-1}\prod_{1\le q, s\le m}
\left(1-\frac{1}{4^{r-p-1}}\right)^2
\ge \prod_{k=1}^{m-2}\left(1-4^{-k}\right)^{2m^3}\\
&=&\exp\left(-2m^3 \sum_{k=1}^{m-2}\sum_{l=1}^\infty
\frac{4^{-kl}}{l}\right)\\
&\ge& \exp\left(-2m^3 \sum_{l=1}^\infty
\frac{1}{(4^{l}-1)l}\right)\ge \exp\left(-2m^3 \sum_{l=1}^\infty \frac{1}{3^ll}\right)\\
&=&\left(2/3\right)^{2m^3}.
\eqn
Therefore, we have \bqn
\prod_{1\le i<j\le
n}\left(\frac{\delta_j-\delta_i}{\delta_j+\delta_i}\right)^2&=&\prod_{(i,j)\in
A} \cdot \prod_{(i,j)\in B}\cdot\prod_{(i,j)\in
C}\left(\frac{\delta_j-\delta_i}{\delta_j+\delta_i}\right)^2
\ge (60e)^{-2m^3}.\eqn
On the other hand, it is not difficult to see that
\bqn
2^n\prod_{i=1}^n\delta_i&=&2^{m^2}\prod_{q=1}^m\prod_{p=0}^{m-1}4^{p+m}(m+q)
<2^{m^2}\cdot
4^{m^2(m-1)/2+m^3}(2m)^{m^2}\\
&=&4^{3m^3/2+m^2/2+m^2\log_4 m}<4^{2m^3}\eqn
for $m>1$. Therefore,
$$\det(a_{ij})
=\left(2^n\prod_{i=1}^n\delta_i\right)^{-1} \cdot\prod_{1\le i<j\le
n}\left(\frac{\delta_j-\delta_i}{\delta_j+\delta_i}\right)^2\ge
(240e)^{-2m^3}.$$
\end{proof}

Now combining the two lemmas above, and using the estimate in
(\ref{4.2}),
we obtain
$$\det(\Sigma)\ge (240e)^{-2m^3}-m^2\cdot e^{-2m4^m}\ge e^{-16m^3}$$
provided that $m$ is large enough. Plugging into
(\ref{determinant}), we have \bqn \PP\left(\sup_{t\ge
0}|Y(t)|<\eps\right) &\le& e^{8m^3} \eps^{m^2}.\eqn
Minimizing the right-hand side by choosing $m\approx |\log \eps|/12,$
we obtain \bqn \PP\left(\sup_{t\ge 0}|Y(t)|<\eps\right) &\lesssim&
\exp\left(-(432)^{-1} |\log \eps|^3\right).\eqn Statement (ii)
of Theorem~\ref{theorem} follows by applying
Corollary~\ref{cor}. At the same time, we also finished the proof of Theorem 1.2.\\

\bibliographystyle{ims}
\bibliography{lap}

\end{document}